\title{Curious Pivot Points of Least-Squares Regression}
\author{Sable Levy}
                  
\date{\today}
\documentclass{article}
\usepackage{graphicx}
\usepackage[utf8]{inputenc}
\usepackage{cancel}
\usepackage{upgreek}
\usepackage{amsmath}
\usepackage{tikz}
\usepackage{amsthm}
\usepackage[T1]{fontenc}
\usepackage{upquote}
\usepackage{amsfonts}
\usepackage{caption}
\numberwithin{equation}{section}
\usepackage[compact]{titlesec}

\newtheorem{corollary}{Corollary}
\newtheorem{proposition}{Proposition}
\newtheorem{conj}{Conjecture}
\newtheorem{definition}{Definition}

\begin{document}
\newpage
\maketitle

\begin{abstract}
    It has been shown that for a given set of points in a plane, the least-squares regression line pivots about a fixed point when any single point in the set is repeated.  We consider what happens when more than one point is repeated. Geometrically, we describe the regions where pivot points may lie under all possible combinations of repetitions. The underlying framework of this pivoting is explored, yielding new open-ended questions. \end{abstract}



\section{Introduction}
Seldom does the underlying geometry of least-squares regression provoke great intrigue among us. But what if there is more than meets the eye? There is, after all, something seemingly curious about the pivoting of regression lines under the condition of datum repetition, as shown by Carl Lutzer~\cite{Lutzer}. Since this little-known phenomenon has yet to be explored in depth, let us proceed to investigate. 
\subsection{The regression line pivots}
The mathematical basis for pivoting is the method of least squares, which gives the line that minimizes $\mathcal{E}$, the sum of squared differences between the estimated and actual $y$ values for given $x$ values. In order to determine this line for a set $S= \left\lbrace{S_1,\ldots,S_n} \mid S_i = (x_{i}, y_{i})\right\rbrace$ of points in a plane,
we can construct a system of equations from their coordinates in the form $A\textbf{p}=\textbf{y}$: \[
\left[ \begin{array}{ccccc}
		x_{1} & 1\\
        x_{2} & 1\\
\vdots & \vdots\\
x_{n} & 1\\

		\end{array} \right] \left [ \begin{array}{c} 
								m\\
								b\\
															\end{array} \right ] = \left [ \begin{array}{c} 
							 	y_{1}\\
                                y_{2}\\
								\vdots\\
                                y_{n}\\
								
							\end{array} \right ],
\]
where $m$ and $b$ are the slope and $y$-intercept of the line, respectively.

                 Following the approach of \cite{Lutzer}, by repeating $S_1$ $k$ times the system of equations becomes
\begin{equation}\label{k+1}
\left[\begin{array}{ccccc}x_1 & 1\\
        \vdots & \vdots\\
        x_1 & 1\\
        x_{2} & 1\\
        x_{3} & 1\\
\vdots & \vdots\\
x_{n} & 1\\ \end{array}\right]
\left[\begin{array}{c}m \\ b \end{array}\right]=
\left[\begin{array}{c}y_1\\
                                \vdots\\
                                y_1\\
                                y_{2}\\
                                y_{3}\\
								\vdots\\
                                y_{n}\\
		\end{array}\right]
\begin{array}{l}
\left.\vphantom{\begin{matrix} 0 \\ \vdots\\ 0
\end{matrix}}\right\}k+1\text{ rows} \\ \left.
\vphantom{\begin{matrix} b_0\\[0.45cm]  \ddots\\ b_0\ \end{matrix}} \right.\end{array}
\end{equation}
We can find \textbf{p} by left-multiplying by the transpose of the coefficient matrix, which produces the normal equation $A^{T}A$\textbf{p}=$A^{T}$\textbf{y}:
\[
\left[ \begin{array}{ccccc}
		kx_{1}^2 + \sum\limits_{j=1}^n x_{j}^2 & kx_{1} + \sum\limits_{j=1}^n x_{j}\\
        \\\
        kx_{1} + \sum\limits_{j=1}^n x_{j} & n+k\\
      
		\end{array} \right] \left [ \begin{array}{c} 
								m\\
								b\\
															\end{array} \right ] = \left [ \begin{array}{c} 
							 	kx_{1}y_{1} + \sum\limits_{j=1}^n x_{j}y_{j}\\
                                \\
                                ky_{1} + \sum\limits_{j=1}^n y_{j}\\
								
							\end{array} \right ].
\]
Solving this provides the values of \textit{m} and \textit{b} that minimize $ \mathcal{E}$, giving us the least-squares solution.

For simplicity, we translate the points in $S$ so that the repeated point $S_1$ is at the origin.  This produces $S_1$-centric coordinates, for which every $S_j \in S$ is expressed in terms of its displacement from $S_1$. In performing this translation nothing is lost geometrically because the slope of the regression line is determined by the position of the points relative to each other. In general, a point $(x_j$, $y_j)$ has $S_i$-centric coordinates $(\upchi_j$, $\upgamma_j)$ where $\upchi_j=x_j-x_i$ and $\upgamma_j= y_j-y_i$.  The $S_1$-centric analog of \eqref{k+1} 
has the normal equation 
\[
\left[ \begin{array}{ccccc}
		\sum\limits_{j=1}^n \upchi_{j}^2 & \sum\limits_{j=1}^n \upchi_{j}\\
        \\\
\sum\limits_{j=1}^n \upchi_{j} & n+k\\
      
		\end{array} \right] \left [ \begin{array}{c} 
								m\\
								b\\
															\end{array} \right ] = \left [ \begin{array}{c} 
							 	 \sum\limits_{j=1}^n \upchi_{j}\upgamma_{j}\\
                                \\
                                 \sum\limits_{j=1}^n \upgamma_{j}\\
								
							\end{array} \right ].
\]

Comparing the first components of the left-hand and right-hand sides gives
$$m\left(\sum\limits_{j=1}^n \upchi_{j}^2\right) + b\left(\sum\limits_{j=1}^n \upchi_{j}\right)= \sum\limits_{j=1}^n \upchi_{j}\upgamma_{j},$$
which we can divide by $\sum\limits_{j=1}^n \upchi_{j}$ to derive $m\tilde{\upchi}+b=\tilde{\upgamma}$ where
$$
\tilde{\upchi}=\frac{\sum\limits_{j=1}^n{\upchi_j^2}}{\sum\limits_{j=1}^n{\upchi_j}}, \indent \tilde{\upgamma}=\frac{\sum\limits_{j=1}^n{\upchi_j\upgamma_j}}{\sum\limits_{j=1}^n{\upchi_j}}.
$$

Since $k$ is absent from these equations, $(\tilde{\upchi}$, $\tilde{\upgamma})$ is on the regression line regardless of how many times $S_1$ is repeated. Intuitively, when a point is repeated, the upgraded regression line moves toward it, if only slightly. 
In order for the regression line to move toward $S_1$ and yet always pass through $(\tilde{\upchi}$, $\tilde{\upgamma})$, it must pivot on this point. And so we have a pivot point.

\section{Pivot point remarks}
\begin{definition}\label{P-centricPivot} Let $\uprho_{i} = (\tilde{\upchi}_i, \tilde{\upgamma}_i)$ be the $S_i$-centric pivot point corresponding to any $S_i \in S$, where
$$\tilde{\upchi}_i = \frac{\sum\limits_{j=1}^n{\upchi_j^2}}{\sum\limits_{j=1}^n{\upchi_j}}, \indent \tilde{\upgamma}_i=\frac{\sum\limits_{j=1}^n{\upchi_j\upgamma_j}}{\sum\limits_{j=1}^n{\upchi_j}}.$$

\end{definition}

\begin{definition}\label{Def1} Let $P_{i} = (\tilde{x}_i, \tilde{y}_i)$ be the pivot point of $S_i$ expressed in terms of the usual Cartesian system. 
$$\tilde{x}_i=x_{i}+\frac{\sum\limits_{j=1}^n{(x_j-x_i)^2}}{\sum\limits_{j=1}^n{(x_j-x_i)}}, \indent  \tilde{y}_i=y_{i}+\frac{\sum\limits_{j=1}^n{(x_j-x_i)(y_j-y_i)}}{\sum\limits_{j=1}^n({x_j-x_i)}}.$$
\end{definition}
This concludes our summary of the results of \cite{Lutzer}. Observe that if $\sum\limits_{j=1}^n \upchi_{j}=0$, then the pivot point is undefined. In such a case, we may consider the regression line to be pivoting at $\infty$. With this exception, a pivot corresponding to each point in $S$ lies on every regression line. If some $S_i$ is not repeated, $P_i$ represents where the regression line \textit{would} pivot.

\begin{definition}\label{reps}
Let $k_r \in \mathbb{N}$ be the number of repetitions of any $S_r \in S$.
\end{definition}

\begin{proposition}\label{ConvergetoR}
   The pivot points corresponding to each of the non-repeated points in $S$ converge to $S_r$ as $k_r \to \infty$.
\end{proposition}

\begin{proof}

When $S_r$ is repeated $k_r$ times, the $S_i$-centric pivot point $\uprho_i$ of any $S_i \in S,$  $i \neq r$ is given by $$\tilde{\upchi}_i=\frac{k_r\upchi_r^2+\sum\limits_{j=1}^n{\upchi_j^2}}{k_r\upchi_r+\sum\limits_{j=1}^n{\upchi_j}}, \indent \tilde{\upgamma}_i=\frac{k_r\upchi_r\upgamma_r+\sum\limits_{j=1}^n{\upchi_j\upgamma_j}}{k_r\upchi_r+\sum\limits_{j=1}^n{\upchi_j}}.$$
It follows that
$$\lim_{k_r\to\infty}\tilde{\upchi}_i =\lim_{k_r\to\infty}\frac{k_r\upchi_r^2+\sum\limits_{j=1}^n{\upchi_j^2}}{k_r\upchi_r+\sum\limits_{j=1}^n{\upchi_j}} = \upchi_r , \indent \lim_{k_r\to\infty} \tilde{\upgamma}_i= \lim_{k_r\to\infty}\frac{k_r\upchi_r\upgamma_r+\sum\limits_{j=1}^n{\upchi_j\upgamma_j}}{k_r\upchi_r+\sum\limits_{j=1}^n{\upchi_j}} = {\upgamma}_r$$ 
\end{proof}
Figure 1 illustrates what happens when repeating the black point. Pivot points are given by the smaller point of matching color; those corresponding to each of the non-repeated points move linearly in the direction of the black point, though not necessarily monotonically.

\begin{figure}[h!]
\centering
\includegraphics[width=4.7in]{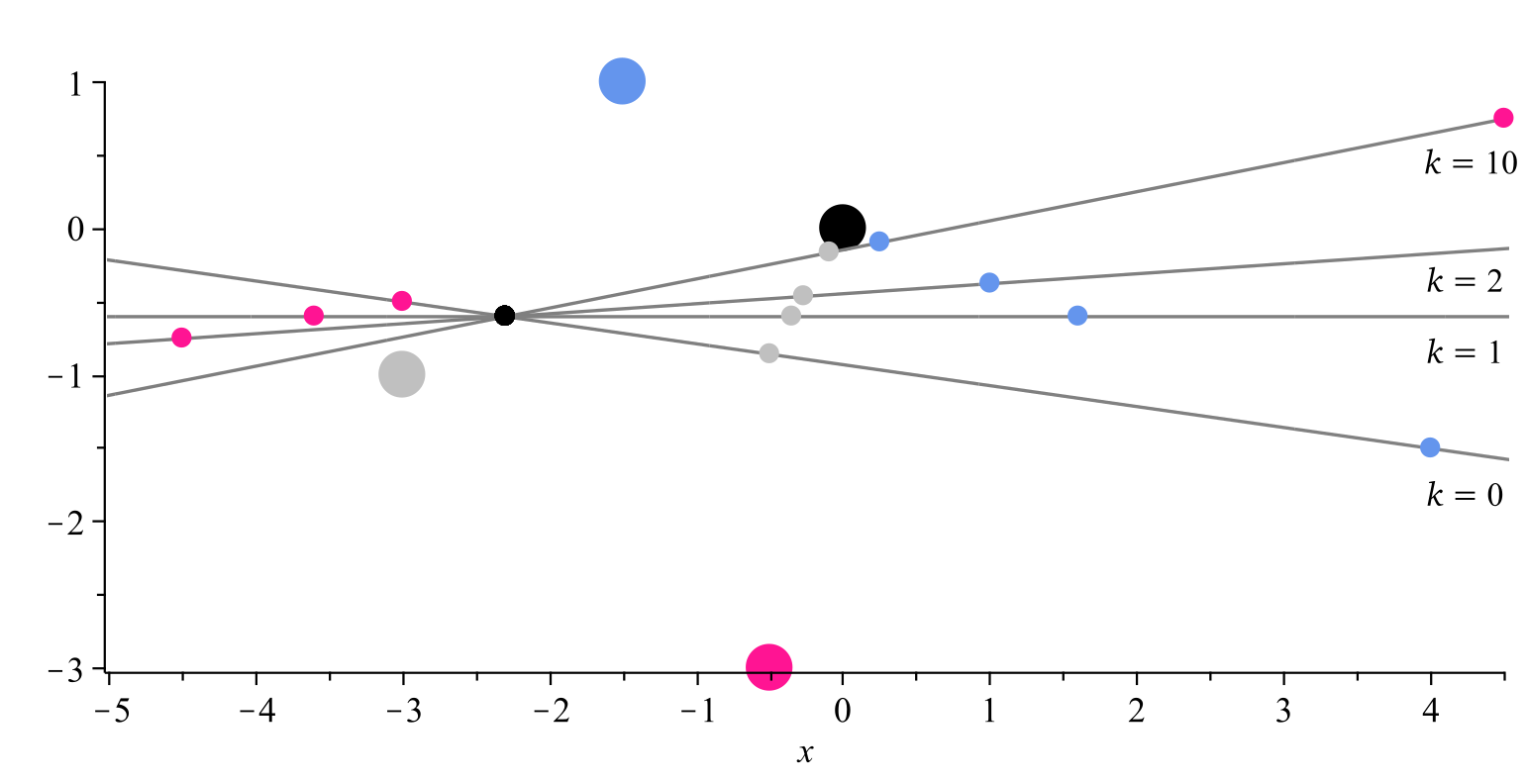}
  
  \caption{Repeating the black point \textit{k} times. Each pivot point represents where the regression line {would} pivot if its corresponding original point were repeated.}
  \label{fig:r4}
\end{figure}

To better understand the relationship between a point and its pivot, consider three points,
\begin{equation}\label{ABC}
\left[ \begin{array}{c}
		S_1\\
        S_2\\
        S_3\\
		\end{array} \right ]    =   
        \left[ \begin{array}{ccccc}
		x_1 & y_1\\
        x_2 & y_2\\
        x_3 & y_3\\
		\end{array} \right] =         \left[ \begin{array}{ccccc}
		0 & 0\\
        \upchi_2 & {\upgamma}_2\\
        \upchi_3 & {\upgamma}_3\\
		\end{array} \right] , \indent x_{1} < x_{2} < x_{3}, \end{equation}
For simplicity, we define the points to be in ascending order on the $x$-axis; no assumptions will be made about the $y$-coordinates. We can perform the transformation 
$$T(\upchi_i)=\frac{\upchi_i-\frac{(\upchi_2+\upchi_3)}{2}}{\frac{\upchi_3-\upchi_2}{2}}$$
which gives

\[T\left(
\left [ \begin{array}{c}
\upchi_1 \\ 
\upchi_2\\ 
\upchi_3\\
\end{array} \right ]\right)  = \left [ \begin{array}{c}
\frac{\upchi_2+\upchi_3}{\upchi_2-\upchi_3} \\ 
-1\\
1\\
\end{array} \right ]. \]
We can see how $\frac{\upchi_2 + \upchi_3}{\upchi_2 - \upchi_3}$ compares to $\tilde{\upchi}_1$, the $x$-coordinate of $S_1$'s pivot point, when performing the same transformation

$$\tilde{\upchi}_1=\frac{\upchi_2^2+\upchi_3^2}{\upchi_2+\upchi_3}, \indent  \indent T(\tilde{\upchi}_1) = \frac{\frac{\upchi_2^2+\upchi_3^2}{\upchi_2+\upchi_3}  - \frac{\upchi_2 + \upchi_3}{2}}{\frac{\upchi_3-\upchi_2}{2}} = \frac{\upchi_3-\upchi_2}{\upchi_2+\upchi_3},$$
which is the negative reciprocal of $S_1$'s transformed $x$-coordinate! We can thus describe the relationship between $x_i$ and $\tilde{x}_i$.  In general, $\tilde{x}_i$ is a {weighted negative inversion} of $x_i$.
Furthermore, $\tilde{y}_i$ is determined by where $\tilde{x}_i$ falls on the regression line.


\section{Barycentric coordinates}
Let us say a few words about a coordinate system that will help us consider the effects of repeating multiple points in $S$.
Barycentric coordinates were introduced by Möbius in his 1827 publication, \textit{Der Barycentrische Calcul}. By attaching weights $\lambda_1,\lambda_2,\lambda_3$ to the respective vertices of a triangle $\triangle S_1S_2S_3$, we can express the position of a point in a plane as a linear combination of those vertices~\cite{Ungar}. We include the condition that the $\lambda_i$ are scaled so that $\lambda_1+\lambda_2+\lambda_3 =1$. 

Without loss of generality, $\lambda_1=0$ along the edge through $S_2$ and $S_3$ because any point along this edge is a linear combination of $S_2$ and $S_3$. Furthermore, $\lambda_i$ is consistent along lines parallel to the edge opposite $S_i$. For our purposes, this is valuable because it provides a computationally efficient way of determining whether a point is located within a triangle.  Specifically, a point $(x,y)$ lies inside (or on an edge of) $\triangle S_1S_2S_3$ if it can be expressed as $(\lambda_1,\lambda_2,\lambda_3)$ where ${x}=\lambda_1x_1+\lambda_2x_2+\lambda_3x_3$ and
${y}=\lambda_1y_1+\lambda_2y_2+\lambda_3y_3$ such that $\lambda_1+\lambda_2+\lambda_3 =1$, $\lambda_i \in [0,1].$
If some $\lambda_i$ $\notin$ [0,1] then the point lies outside the triangle. Furthermore, the signs of the barycentric coordinates of a given point indicate precisely which region of the plane that point is in ~\cite{Nikkhoo}. See Figure \ref{fig:barysigns} for an example.

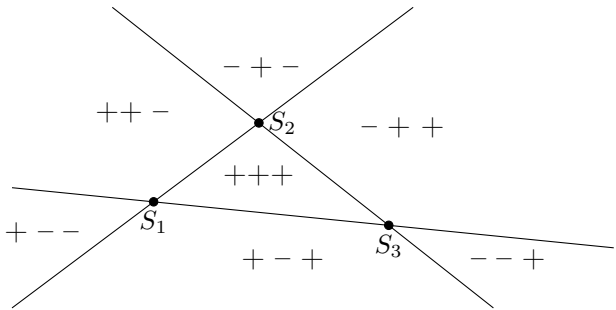
\begin{figure}[h!]
\centering
\begin{tikzpicture}[scale=.8]
    
    \draw (0,0) coordinate (a_1) -- (8/1.2,6/1.2) coordinate (a_2);
    \draw (2/1.2,6/1.2) coordinate (b_1) -- (8,0) coordinate (b_2);
     \draw (0,2) coordinate (c_1) -- (10,1) coordinate (c_2);
    \coordinate (d) at (intersection of a_1--a_2 and b_1--b_2);
    \node [right] at (d) {$S_2$};
    \coordinate (e) at (intersection of a_1--a_2 and c_1--c_2);
        \node [below] at (e) {$S_1$};
    \coordinate (f) at (intersection of c_1--c_2 and b_1--b_2);
        \node [below] at (f) {$S_3$};
    \node at (.5,1.25) {\textbf{+ {--} {--}}};
     \node at (4.1,2.2) {\textbf{+++}};
     \node at (2,3.25) {\textbf{++ --}};
     \node at (4.15,4.) {\textbf{-- + --}};
     \node at (4.5,.8) {\textbf{+ -- +}};
     \node at (8.25,.8) {\textbf{-- -- +}};
     \node at (6.5,3) {\textbf{-- + +}};
     \draw[fill] (d) circle [radius=0.07];
    \draw[fill] (e) circle [radius=0.07];
    \draw[fill] (f) circle [radius=0.07];
\end{tikzpicture}
  \caption{The signs of $\lambda_1, \lambda_2, \lambda_3$ specify which region of the plane a point is in.}
  \label{fig:barysigns}
\end{figure}

\section{Repeating multiple points}
Let us broaden the definition of $\uprho_i$ by incorporating $k_j$ repetitions of each $S_j \in S$.

\begin{definition}Given $k_j$ repetitions of each $S_j \in S$ with $k_j + 1 = \delta_j$, let $\uprho_i = (\tilde{\upchi}_i, \tilde{\upgamma}_i)$ where
$$\tilde{\upchi}_i = \frac{\sum\limits_{j=1}^n{\delta_j\upchi_j^2}}{\sum\limits_{j=1}^n{\delta_j\upchi_j}}, \indent \tilde{\upgamma}_i=\frac{\sum\limits_{j=1}^n{\delta_j\upchi_j\upgamma_j}}{\sum\limits_{j=1}^n{\delta_j\upchi_j}}.$$
\end{definition}

To begin, suppose $S$ contains only four points. 

\[
\left [ \begin{array}{c}
		S_1\\
        S_2\\
        S_3\\
        S_4\\
		\end{array} \right ]
     =   
        \left[ \begin{array}{ccccc}
        x_{1} & y_1\\
		x_{2} & y_2\\
        x_{3} & y_3\\
        x_4 & y_4\\

		\end{array} \right]
        = \left[ \begin{array}{ccccc}
        \upchi_1 & \upgamma_1\\
		\upchi_2 & \upgamma_2\\
        \upchi_{3} & \upgamma_3\\
        0 & 0\\

		\end{array} \right], \indent x_{1} \leq x_{2} \leq x_{3} \leq x_4.\] 

\begin{proposition}With four points, the pivots points corresponding to an outermost point on the $x$-axis are bounded by the triangle whose vertices are the three other points.

\end{proposition} 
\begin{proof}
We have $\uprho_4$ = $(\tilde{\upchi}_4$, $\tilde{\upgamma}_4)$ where
$$
    \tilde{\upchi}_{4}=\frac{\delta_1\upchi_1^2+\delta_2\upchi_2^2+\delta_3\upchi_3^2}{\delta_1\upchi_1+\delta_2\upchi_2+\delta_3\upchi_3}, \indent \tilde{\upgamma}_{4}=\frac{\delta_1\upchi_1\upgamma_1+\delta_2\upchi_2\upgamma_2+\delta_3\upchi_3\upgamma_3}{\delta_1\upchi_1+\delta_2\upchi_2+\delta_3\upchi_3}.$$
 Expressing $\uprho_4$ using barycentric coordinates,
 
\[
\left[ \begin{array}{ccccc}
		\upchi_{1} & \upchi_{2} & \upchi_3\\
        \upgamma_{1} & \upgamma_{2} & \upgamma_3\\

1 &   1 & 1\\

		\end{array} \right] \left [ \begin{array}{ccccc}
		\lambda_1 \\
        \lambda_2\\
        \lambda_3\\
		\end{array} \right ] = \left [ \begin{array}{ccccc}
		\tilde{\upchi}_4 \\
        \tilde{\upgamma}_4\\
        1\\
		\end{array} \right ] \]  
\[ \left [ \begin{array}{ccccc}
		\lambda_1 \\
        \lambda_2\\
        \lambda_3\\
		\end{array} \right ] = \left[ \begin{array}{ccccc}
		\upchi_{1} & \upchi_{2} & \upchi_3\\
        \upgamma_{1} & \upgamma_{2} & \upgamma_3\\

1 &   1 & 1\\

		\end{array} \right]^{-1} \left [ \begin{array}{ccccc}
	\frac{\delta_1\upchi_1^2+\delta_2\upchi_2^2+\delta_3\upchi_3^2}{\delta_1\upchi_1+\delta_2\upchi_2+\delta_3\upchi_3}  \\[0.1in]
        \frac{\delta_1\upchi_1y_1+\delta_2\upchi_2y_2+\delta_3\upchi_3y_3}{\delta_1\upchi_1+\delta_2\upchi_2+\delta_3\upchi_3}\\[0.1in]
        1\\
		\end{array} \right ]
\]  

\[
 \left [ \begin{array}{ccccc}
		\lambda_1 \\
        \lambda_2\\
        \lambda_3\\
		\end{array} \right ] = \left [ \begin{array}{ccccc}
		\frac{\delta_1\upchi_1}{\delta_1\upchi_1+\delta_2\upchi_2+\delta_3\upchi_3} \\[0.1in]
        \frac{\delta_2\upchi_2}{\delta_1\upchi_1+\delta_2\upchi_2+\delta_3\upchi_3}\\[0.1in]
        \frac{\delta_3\upchi_3}{\delta_1\upchi_1+\delta_2\upchi_2+\delta_3\upchi_3}\\[0.1in]
		\end{array} \right ].
\]

Having made the points $S_4$-centric, each $\upchi_i$ $\leq 0$, from which it follows that each $\lambda_i$ is positive. Therefore the pivot points corresponding to $S_4$ lie within $\triangle S_1S_2S_3$.  
Translating the points in $S$ so that $S_1$ is at the origin, we likewise find that $S_1$'s pivot points have strictly positive barycentric coordinates, i.e., that the pivot points corresponding to $S_1$ lie within $\triangle S_2S_3S_4$. \end{proof}

\begin{proposition}
The pivots points corresponding to $S_2$ lie in the unbounded regions $+--$ and $-++$ of $\triangle S_1S_3S_4$ while those of $S_3$ lie in the unbounded regions $++-$ and $--+$ of $\triangle S_1S_2S_4$.
\end{proposition}
\begin{proof}
We know that $\uprho_2 = (\lambda_1, \lambda_3, \lambda_4)$ 
where $\lambda_i = \frac{\delta_i\upchi_i}{\delta_1\upchi_1+\delta_3\upchi_3+\delta_4\upchi_4}$. Since $\upchi_1 \leq \\[0.05in] 0 \leq \upchi_3, \upchi_4$ in $S_2$-centric coordinates, it follows that $\lambda_3$ and $\lambda_4$ must have the same sign.  Since the $\lambda_i$ sum to 1, $\lambda_1$ must have the opposite sign. Similarly, $\uprho_3 = (\lambda_1, \lambda_2, \lambda_4)$ where $\lambda_i = \frac{\delta_i\upchi_i}{\delta_1\upchi_1+\delta_2\upchi_2+\delta_4\upchi_4}$. Given  $S_3$-centric coordinates,\\[0.05in] $\upchi_1, \upchi_2 \leq 0 \leq \upchi_4$, from which it follows that $\lambda_1$ and $\lambda_2$ share a sign, opposite to that of $\lambda_4$.
\end{proof}

\begin{corollary}With four points, a pivot points cannot lie in the $+-+$ or $-+-$ regions of any triangle whose vertices are a subset of $S$.
\end{corollary} 
\begin{proof}
This result follows directly from propositions 2 and 3, which jointly prove that five of the seven regions shown in Figure \ref{fig:barysigns} are candidates for pivot points. The $+-+$ and $-+-$ regions are by exhaustion ineligible.
\end{proof}
Figure 3 provides a visual example of the four-point case, though only a small portion of the unbounded regions are shown.
\begin{figure}[h!]
\centering
  \includegraphics[width=4.8in]{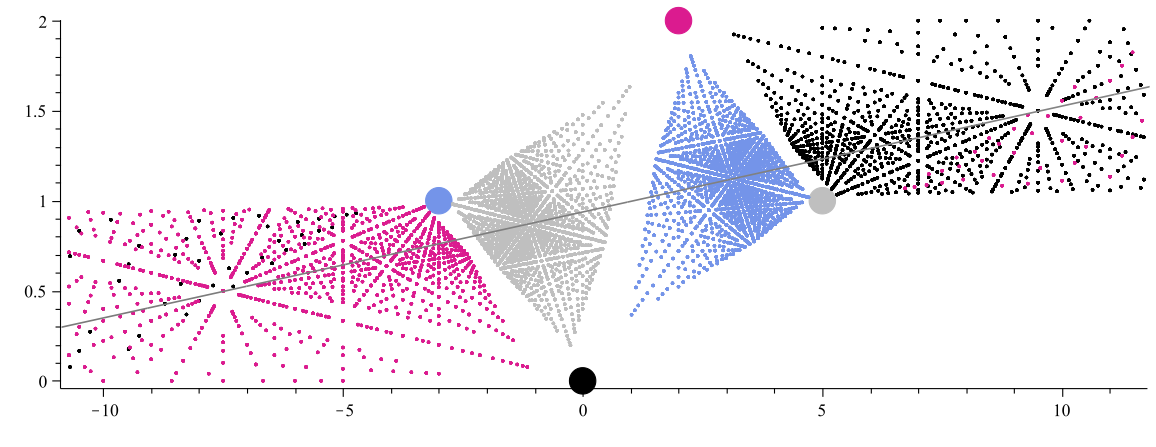}
  \caption{Pivot points for every combination of up to twelve repetitions of each point in a set of four, shown with the original regression line.}
  \label{fig:rbary}
\end{figure}


\subsection{The three-point case}
We can use barycentric coordinates to understand the pivoting behavior of a set of three, as given by \eqref{ABC}. Having lost a degree of freedom, the pivot points corresponding to any $S_i$ will lie on the line through the other two points; anywhere on that line, a pivot point is a linear combination of those two vertices. (The non-existent 3rd vertex always has a barycentric coordinate of 0.) The pivot points corresponding to an outermost point on the $x$-axis are bounded by the line segment between the other two points, while those corresponding to the inner point can lie only on the unbounded parts of that line.  An example of this is provided in Figure \ref{fig:r3}.
\begin{figure}[h!]
\centering
  \includegraphics[width=3.7in]{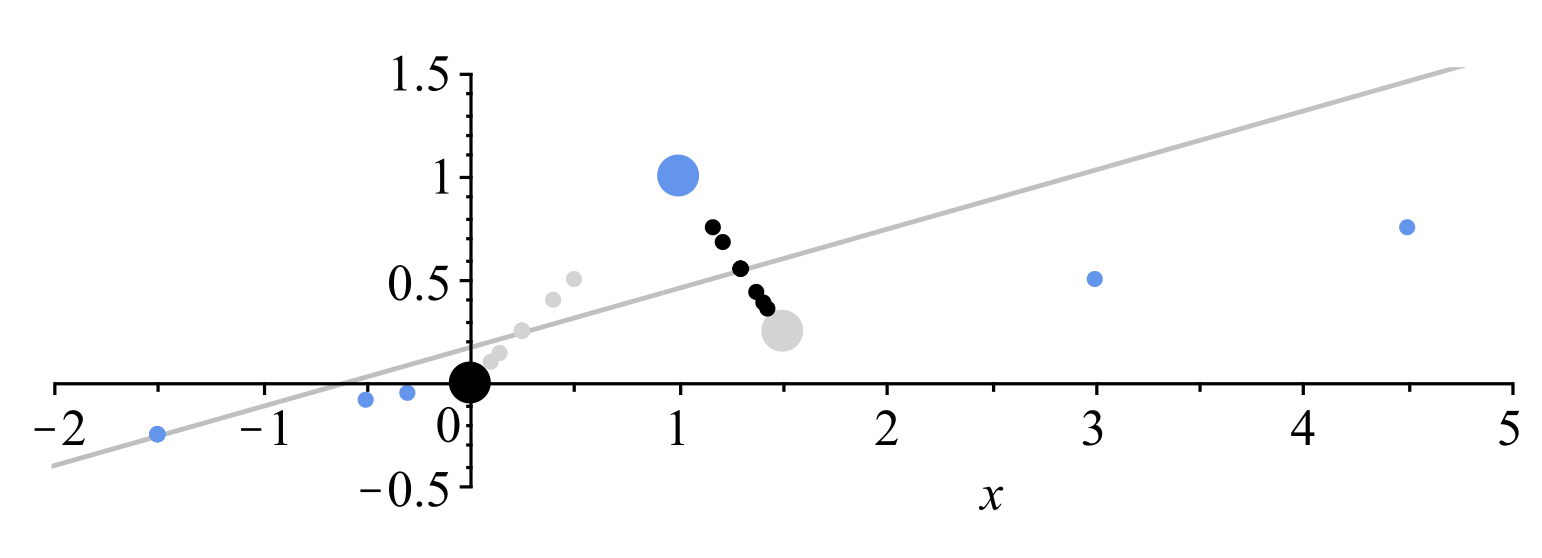}
  \caption{
Three points and their original regression line; an outermost point on the $x$-axis has pivots that are bounded by the line segment between the other two points; pivots corresponding to the inner point have the inverse property of occurring on either side of that line segment---but not within it.
}
  \label{fig:r3}
\end{figure}
\subsection{Generalizing to $n$ points}


Mean value coordinates are a generalization of barycentric coordinates to convex polytopes with more than three vertices~\cite{Floater}, for which the topology of triangulation is not unique~\cite{Ghosh}.
Within the interior of a polygon, mean value coordinates will be strictly positive, as with barycentric coordinates and triangle interiors. Since extreme points on the $x$-axis have pivot points that are a positive weighted average of the non-corresponding points in $S$, with $n$ points the pivot points corresponding to an outermost point on the $x$-axis are bounded by the convex hull of the other $n-1$ points.

Furthermore, with $n$ points, a pivot point will travel linearly toward a repeated point along the path that does not cross the $x$-coordinate of the repeated point’s pivot point. This phenomenon, which maps out the regions where pivot points can lie under all possible combinations of repetitions, can be observed in Figure 1.

\section{Pseudopivoting}
What happens if every point in $S$ is collinear? Naturally, we will lose the \textit{appearance} of pivoting; no matter how many times we repeat any combination of points, the regression line will always be the line through every point. Nonetheless, the formulas given in Definition 2 will still yield coordinates---pseudopivots, we might say.  Observe that while $\tilde{x}_i$ depends only on $x$-coordinates,  $\tilde{y}_i$ depends on both. Let us disregard the $y$ coordinates for a moment, and imagine that our coordinates lie on a horizontal line.  We know that if every $y_i$ is equal, then $\tilde{y} = y_i$ and there is nothing more to say about it:

$$\tilde{y}=y_1+\frac{(y_1-y_1)[{(x_2-x_1)+\ldots+(x_n-x_1)}]}{{(x_2-x_1)+\ldots+(x_n-x_1)}} = y_1.$$
\\The $x$-coordinates, however, will not cease to be interesting.  Indeed, they are candidates for iteration: If we consider only the $x$'s, then what are the pseudopivot coordinates of the pseudopivot coordinates? Let us consider a simple case. A set $\left\lbrace{a,b,c}\right\rbrace$ of real numbers has pseudopivots $\left\lbrace{{a}_1,{b}_1,{c}_1}\right\rbrace$, which we derive from the formula for $\tilde{x}_i$.
$${a}_1= a + \frac{(b-a)^2+(c-a)^2}{(b-a)+(c-a)} = 
\frac{b^2+c^2-a(b+c)}{b+c-2a}$$
$${b}_1= b + \frac{(a-b)^2+(c-b)^2}{(a-b)+(c-b)} = 
\frac{a^2+c^2-b(a+c)}{a+c-2b}$$
$${c}_1= c + \frac{(a-c)^2+(b-c)^2}{(a-c)+(b-c)} = 
\frac{a^2+b^2-c(a+b)}{a+b-2c}$$
In general, the respective formulas for the $nth$ psuedopivot of $a$, $b$, $c$ are:
$${a}_n=\frac{{b}_{n-1}^2+{c}_{n-1}^2-{a}_{n-1}({b}_{n-1}+{c}_{n-1})}{{b}_{n-1}+{c}_{n-1}-2{a}_{n-1}}$$
$${b}_n=\frac{{a}_{n-1}^2+{c}_{n-1}^2-{b}_{n-1}({a}_{n-1}+{c}_{n-1})}{{a}_{n-1}+{c}_{n-1}-2{b}_{n-1}}$$
$${c}_n=\frac{{a}_{n-1}^2+{b}_{n-1}^2-{c}_{n-1}({a}_{n-1}+b_{n-1})}{{a}_{n-1}+{b}_{n-1}-2{c}_{n-1}}$$

Figure \ref{fig: X's up to 6} provides a visual example, with $n$ on the vertical axis. \begin{figure}[h!]
\centering
  \includegraphics[width=4.5in]{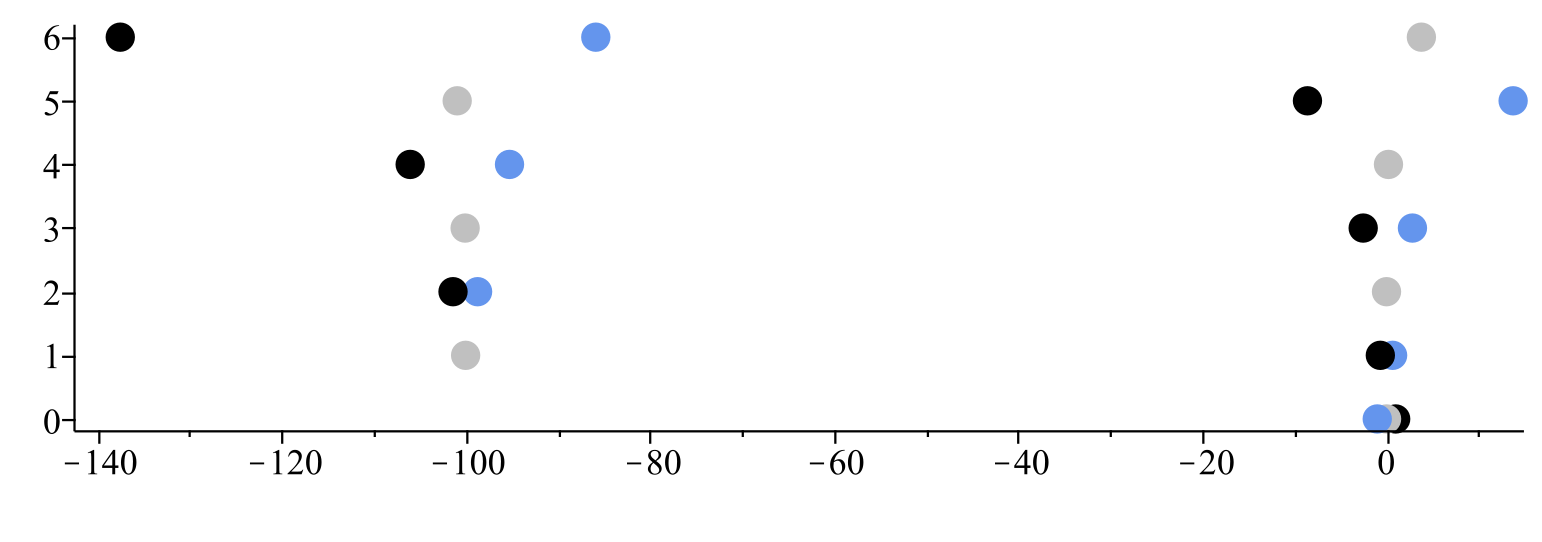}
  \caption{Six iterations of psuedopivots for the set $\left\lbrace{-1,.01,1}\right\rbrace$.}
  \label{fig: X's up to 6}
\end{figure} 

A distinct pattern can be observed. Moving from one iteration to the next, the outermost points swap positions while the inner point moves to its opposite extreme. As mentioned, if a number in the set is the average of the others, then it pivots at $\infty$. These averages can be thought of as bifurcation thresholds---if the inner coordinate is greater than the average, its psuedopivot coordinate will be the minimum of the successive set, while the pseudopivot of a below-average inner coordinate will be the successive max.

Consider the black point, for instance.  While it is the original maximum, it is not the max for any set up to $n=6$.  As can be observed in Fig. \ref{fig: X's 5 to 8}, it finally becomes a max when $n=8$.  We can see why upon inspection of the $n=7$ set, in which the black point is (for the first time) an inner coordinate that is less than the average of the outer two. 

\begin{figure}[h!]
\centering
  \includegraphics[width=4.5in]{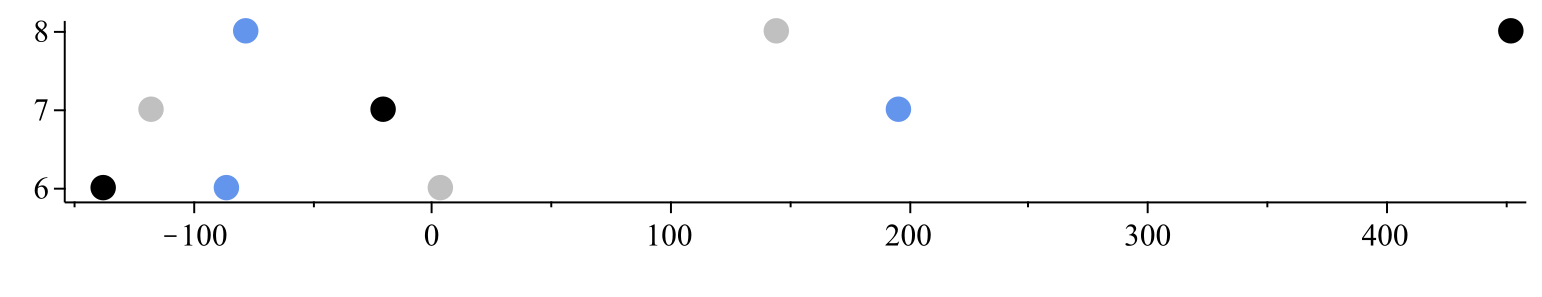}
  \caption{Additional pseudopivot coordinates for $\left\lbrace{-1,.01,1}\right\rbrace$, wherein the black point crosses the bifurcation threshold when $n=7$.}
  \label{fig: X's 5 to 8}
\end{figure}

What happens to $\left\lbrace{a_n,b_n,c_n}\right\rbrace$ as $n\to\infty$?

\begin{conj}
The range of pseudopivot points will increase with each successive iteration ad infinitum.
\end{conj}

 Moreover, we know that the same pattern of permutations will not occur twice in a row. Subject to this, are there sequences of points that have definable patterns?  We leave this question open-ended.


\section{Further questions}
The discussion presented herein leaves a great many questions open for whoever is curious to know. We include some below and invite the reader to add their own.  

\begin{itemize} 
\item In order to plot a continuous rather than discrete repetition variable, can we derive a parameterization of the region where a set of pivot points can lie?   
\item Does every possible sequence of pseudopivot permutations correspond to a set of points?
\item What does pseudopivoting look like in higher dimensions?
\end{itemize}

\section*{Acknowledgements} The author wishes to thank Dr.~Edward Early for his mentorship throughout this project; Dr.~Mitch Phillipson, who also helped; and Noelle Mandell, without whom none of this would have happened.


\end{document}